\documentclass[12pt]{article}
\usepackage[utf8]{inputenc}
\usepackage[dvips]{graphicx}
 \usepackage{amssymb,amsmath}
 \usepackage{amsthm}
\usepackage{epic}
\usepackage{color}
\usepackage{xargs}
\usepackage[left=1.0in,top=1in,right=1.0in,bottom=1in]{geometry}

\usepackage{tikz}
\usetikzlibrary{arrows.meta}
\tikzstyle{vertex}=[circle, draw, inner sep=2pt, minimum size=6pt]
\newcommand{\vertex}{\node[vertex]}

\newtheorem{theorem}{Theorem}
\newtheorem{proposition}[theorem]{Proposition}

\newtheorem{corollary}[theorem]{Corollary}

\numberwithin{claim}{theorem}

\newcommand{\B}{\mathcal{B}}

 \newcommand{\ovd}[1]{\overrightarrow{#1}}

 \newcommand{\omitir}[1]{}
 
 \usepackage[colorinlistoftodos]{todonotes}

\newcommandx{\mm}[2][1=]{\todo[size=\tiny,backgroundcolor=yellow!20,linecolor=red,#1]{#1 \\ \hfill --- Martín M}}

\begin{document}

\title{Lines in quasi-metric spaces with four points\footnote{
Supported by ANID Basal program FB21005 and PAPIIT-UNAM grant IN108121.}}
\author{Gabriela Araujo-Pardo$^1$, Martín Matamala$^2$ and Jos\'e Zamora$^3$\\
\\
  ($1$) Instituto de Matemáticas, Universidad Nacional Aut\'onoma de M\'exico, M\'exico\\
  ($2$) DIM-CMM, CNRS-IRL 2807, Universidad de Chile, Chile\\
  ($3$) Depto. de Matemáticas, Universidad Andres Bello, Chile.}

\maketitle

\begin{abstract} A set of $n$ non-collinear points in the Euclidean plane defines at least $n$ different lines. Chen and Chvátal in 2008 conjectured that the same results is true in metric spaces for an adequate definition of line.
More recently, this conjecture was studied in the context of quasi-metric spaces.

In this work we prove that there is a quasi-metric space on four points $a$, $b$, $c$ and $d$ whose betweenness is $\B=\{(c,a,b),(a,b,c),(d,b,a),(b,a,d)\}$. Then, this space has only three lines none of which has four points. Moreover, we show that the betweenness of any quasi-metric space on four points with this property is isomorphic to $\B$. Since $\B$ is not metric, we get that Chen and Chvátal's conjecture is valid for any metric space on four points.
\end{abstract}

\section{Introduction}

The Chen and Chvátal's conjecture, introduced in 2008 (see \cite{CC}), affirms that every finite metric space $M$ with $n\geq 2$ points has the so called \emph{de Bruijn and Erdös} (DBE) property: 
\begin{equation}
M \text{ has a \emph{line} containing all the points or at least $n$ different \emph{lines}} \label{dbe}\tag{DBE} 
\end{equation}

This conjecture generalizes to metric spaces a well-known result known as de Bruijn and Erdös's Theorem (see \cite{chvatal_2018} for a more detailed account), proved by Erdös in \cite{erdos1943} and generalized by de Bruijn and Erdös in \cite{dbe}, saying that every set of $n\geq 2$ points in the Euclidean plane determines at least $n$ distinct lines unless all the points lay in the same line.

Lines are defined in metric spaces as a natural generalization of its definition in the plane. If $d$ is the metric of a metric space, then given two distinct points $x$ and $y$, the \emph{segment} defined by $x$ and $y$, denoted by $[xy]$, is the set defined by 
$$[xy]=\{z\mid d(x,y)=d(x,z)+d(z,y)\}.$$
Similarly, the \emph{line} defined by $x$ and $y$, denoted by $\ovd{xy}$, is the set defined by 
$$\ovd{xy}=\{z\mid x\in [zy] \lor z\in [xy]\lor y\in [xz]\}.$$

The set of all lines in a metric space $M=(V,d)$ is denoted by $\mathcal{L}(M)$. That is, 
$$\mathcal{L}(M)=\{\ovd{xy}\mid x,y\in V\}.$$

In \cite{AM20}, the DBE property was study in the context of \emph{quasi-metric spaces}, first studied by Wilson in \cite{Wilson1931}, and defined as a pair $(V,d)$ where $d$ is a quasi-metric on $V$, that is, it satisfies $\forall x,y\in V, d(x,y)=0 \iff x=y$,
and $\forall x,y,z\in V, d(x,y)\leq d(x,z)+d(z,y)$. 

As far as we known, the DBE property is valid for metric spaces with distances in $\{0,1,2\}$ \cite{ChCh11, Chvatal2}, with bounded diameter and large enough number of points \cite{metricSpace} and for metric spaces defined  
by chordal graphs \cite{BBCCCCFZ}, distance-hereditary graphs \cite{AK}, graphs such that every induced subgraph is either a chordal graph, has a cut-vertex or a non-trivial module \cite{AMRZ},  
bisplit graphs \cite{BKR}, $(q,q-4)$-graphs \cite{SS} and graphs without induced house or hole \cite{ABMZ}. It is also valid for metric spaces defined by points in the plane with the Manhattan metric, when no two points lay in a horizontal or in a vertical line  \cite{KP2013}. Additionally, any quasi-metric space induced by tournaments or  by bipartite tournaments of diameter at most three \cite{AM20} has the DBE property. Stronger results were proved in \cite{MZ2020}.

In this work we study quasi-metric spaces on four points. It is worth to notice that the number of quasi-metric spaces on a given number of points is infinite since we do not have a restriction on the distances other than that given by the
triangle inequality. This is also true, even when we consider only non-isomorphic quasi-metric spaces.

However, in the study of lines, even some non-isomorphic quasi-metric spaces could be considered as the same. In fact, it is clear that by knowing the segments defined by pairs of distinct points we also know all the lines defined by them. This fact is captured by the \emph{betweenness relation} first studied by Menger \cite{menger1928} in the context of metric spaces. The betweenness of a quasi-metric space $Q$, denoted by $\B(Q)$, is the set of all triples $(x,y,z)$ with $x,y,z$ three distinct points in $Q$ such that $y\in [xz]$. In this work, to ease the presentation, we denote a triple $(x,y,z)$ by the word $xyz$. Then, the betweenness of a quasi-metric space is given by:
$$\B(Q)=\{xyz\in V^3\mid d(x,z)=d(x,y)+d(y,z)\}.$$

When $Q$ has $n$ points we have that $|B(Q)|\leq n(n-1)(n-2)$, and thus the number of betweenness of quasi-metric spaces on $n$ points is at most $2^{n(n-1)(n-2)}$.  

The upper bound $n(n-1)(n-2)$ is too pessimistic since when a triple $xyz$ belongs to $\B(Q)$ we know that the triples $yxz$ and $xzy$ do not belong to $\B(Q)$. Moreover, two quasi-metric spaces $Q=(V,d)$ and $Q'=(V',d')$ whose betweenness relations are different could still be considered as equal when their betweenness relations are (\emph{point-wise}) isomorphic. That is, when  
there is a bijection $f$ between $V$ and $V'$ such that 
$$\forall x,y,z\in V, xyz\in \B(Q) \iff f(x)f(y)f(z)\in \B(Q').
$$

When $n=3$, we can show that there are only five quasi-metric spaces whose betweenness are not isomorphic.
The betweenness of these spaces on the set $\{a,b,c,d\}$ are:
$$\emptyset, \{abc\}, \{abc,cba\}, \{abc,bca\} \text{ and }\{abc,bca,cab\}.$$

To each of them we can associate a quasi-metric space defined by a directed graph (see  Figure \ref{f:3digraphs}). 

If $\B$ is a betweenness of a metric space, then  $xyz\in \B$ implies that $zyx\in \B$. Hence, only the empty set and $\{abc,cba\}$ are betweenness of metric spaces on three points. 

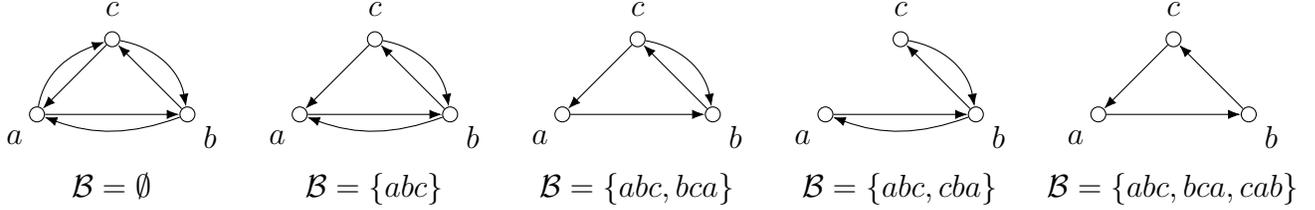
\begin{figure}
\centering
\begin{tabular}{ccccc}
\begin{tikzpicture}[scale=1]
\vertex[circle,  minimum size=4pt](1) at  (-1,0) {};
\vertex[circle,  minimum size=4pt](2) at  (1,0) {};
\vertex[circle,  minimum size=4pt](3) at  (0,1) {};
\node (4) at  (-1.3,-0.3) {$a$};
\node (5) at  (1.3,-0.3) {$b$};
\node (6) at  (0,1.4) {$c$};

\draw[-{Latex}] (1) --
(2);
\draw[-{Latex}] (2) to[out=-160,in=-20] (1);

\draw[-{Latex}] (2) -- (3);
\draw[-{Latex}] (3) to[out=-10,in=100] 
(2);

\draw[-{Latex}] (3) -- 
(1);
\draw[-{Latex}] (1) to[out=80,in=-160] (3);
\end{tikzpicture}
&
\begin{tikzpicture}
\vertex[circle,  minimum size=4pt](1) at  (-1,0) {};
\vertex[circle,  minimum size=4pt](2) at  (1,0) {};
\vertex[circle,  minimum size=4pt](3) at  (0,1) {};
\node (4) at  (-1.3,-0.3) {$a$};
\node (5) at  (1.3,-0.3) {$b$};
\node (6) at  (0,1.4) {$c$};

\draw[-{Latex}] (1) --
(2);
\draw[-{Latex}] (2) to[out=-160,in=-20] (1);
\draw[-{Latex}] (2) -- (3);
\draw[-{Latex}] (3) to[out=-10,in=100] 
(2);

\draw[-{Latex}] (3) -- 
(1);
\end{tikzpicture}
&
\begin{tikzpicture}
\vertex[circle,  minimum size=4pt](1) at  (-1,0) {};
\vertex[circle,  minimum size=4pt](2) at  (1,0) {};
\vertex[circle,  minimum size=4pt](3) at  (0,1) {};
\node (4) at  (-1.3,-0.3) {$a$};
\node (5) at  (1.3,-0.3) {$b$};
\node (6) at  (0,1.4) {$c$};

\draw[-{Latex}] (1) --
(2);

\draw[-{Latex}] (2) -- (3);
\draw[-{Latex}] (3) to[out=-10,in=100] 
(2);
\draw[-{Latex}] (3) -- 
(1);
\end{tikzpicture}
&
\begin{tikzpicture}
\vertex[circle,  minimum size=4pt](1) at  (-1,0) {};
\vertex[circle,  minimum size=4pt](2) at  (1,0) {};
\vertex[circle,  minimum size=4pt](3) at  (0,1) {};
\node (4) at  (-1.3,-0.3) {$a$};
\node (5) at  (1.3,-0.3) {$b$};
\node (6) at  (0,1.4) {$c$};

\draw[-{Latex}] (1) --
(2);
\draw[-{Latex}] (2) to[out=-160,in=-20] (1);

\draw[-{Latex}] (2) -- (3);
\draw[-{Latex}] (3) to[out=-10,in=100] 
(2);

\end{tikzpicture}
& 
\begin{tikzpicture}
\vertex[circle,  minimum size=4pt](1) at  (-1,0) {};
\vertex[circle,  minimum size=4pt](2) at  (1,0) {};
\vertex[circle,  minimum size=4pt](3) at  (0,1) {};
\node (4) at  (-1.3,-0.3) {$a$};
\node (5) at  (1.3,-0.3) {$b$};
\node (6) at  (0,1.4) {$c$};

\draw[-{Latex}] (1) --
(2);

\draw[-{Latex}] (2) -- (3);
(2);

\draw[-{Latex}] (3) -- 
(1);
\end{tikzpicture}
\\
$\B=\emptyset$ & $\B=\{abc\}$ & $\B=\{abc,bca\}$& $\B=\{abc,cba\}$ & $\B=\{abc,bca,cab\}$
\end{tabular}
 \caption{Directed graphs with three vertices defining quasi-metric spaces on three points with non-isomorphic betweenness.}\label{f:3digraphs}
\end{figure}

The lines of these quasi-metric spaces are the following.
\begin{center} 
 \begin{tabular}{|c|c|c|c|c|c|c|c|}
\hline
 $\B$ & $\ovd{ab}$ & $\ovd{ba}$ & $\ovd{ac}$ & $\ovd{ca}$ & $\ovd{bc}$ & $\ovd{cb}$ & $|{\cal L}(Q)|$ \\ \hline
  $\emptyset$ & $\{a,b\}$ & $\{a,b\}$ & $\{a,c\}$ & $\{a,c\}$ & $\{b,c\}$ & $\{b,c\}$ & 3\\ \hline
 $\{abc\}$ & $\{a,b,c\}$ & $\{a,b\}$ & $\{a,b,c\}$ & $\{a,c\}$ & $\{a,b,c\}$ & $\{b,c\}$ & 4\\ \hline
 $\{abc,bca\}$ & $\{a,b,c\}$ & $\{a,b,c\}$ & $\{a,b,c\}$ & $\{a,b,c\}$ & $\{a,b,c\}$ & $\{b,c\}$ & 2\\ \hline
 $\{abc,cba\}$ & $\{a,b,c\}$ & $\{a,b,c\}$ & $\{a,b,c\}$ & $\{a,b,c\}$ & $\{a,b,c\}$ & $\{a,b,c\}$ & 1\\ \hline
 $\{abc,bca,cab\}$ & $\{a,b,c\}$ & $\{a,b,c\}$ & $\{a,b,c\}$ & $\{a,b,c\}$ & $\{a,b,c\}$ & $\{a,b,c\}$ & 1\\ \hline
 \end{tabular}
\end{center}
Hence, all have the DBE property. Moreover, each betweenness on three points can be defined by a quasi-metric space defined by a directed graph. 

We show below a quasi-metric space $Q(4)$, on four points, with distance in $\{0,1,2,3\}$, which does not have the DBE property, showing that the Chen and Chvátal conjecture does not extend to the class of  quasi-metric spaces with distance in $\{0,1,2,3\}$.

However, it could be a notable exception: our main result in this work is that any quasi-metric space on four points whose betweenness is not isomorphic to $\B(Q(4))$ has the DBE property. Moreover, we prove that $\B(Q(4))$ is neither isomorphic to the betweenness of a quasi-metric space with distances in $\{0,1,2\}$, nor to that of a space defined by a directed graph, nor to that of a metric space. As a consequence, any metric space on four points has the DBE property.

The quasi-metric of the space $Q(4)$ and its representation as a weighted directed graph is given in the following table. 

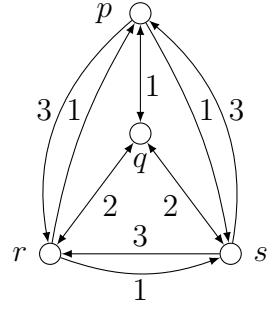
\begin{tabular}[b]{c c c}
\begin{minipage}[b]{6cm}
	\begin{center}
	\begin{tabular}{|c|c|c|c|c|}
	 \hline 
	$d(\cdot,\cdot)$ & $p$ & $s$ & $q$ & $r$ \\
	\hline 
	 $p$ & $0$ & $1$ & $1$ & $3$ \\
	\hline 
	 $s$ & $3$ & $0$ & $2$ & $3$ \\
	 \hline 
	 $q$ & $1$ & $2$ & $0$ & $2$ \\
	\hline
	 $r$ & $1$ & $1$ & $2$ & $0$\\
	 \hline 
	\end{tabular}
	\end{center}
	$Q(4)$: quasi-metric space without Property (DBE).
\end{minipage}
	& \ \ \ & \begin{minipage}[b]{7cm} 
	\begin{center}
\begin{tikzpicture}[scale=0.8]

	\vertex[circle,  minimum size=8pt](1) at  (0,2) {};
	\vertex[circle,  minimum size=8pt](2) at  (-1.5,0) {};
	\vertex[circle,  minimum size=8pt](3) at  (1.5,0) {};
	\vertex[circle,  minimum size=8pt](4) at  (0,4) {};
	\node (5) at  (-0.6,4) {$p$};
	\node () at  (2.0,0) {$s$};
	\node () at  (-2.0,0) {$r$};
	\node () at  (-0.0,1.5) {$q$};
	\node () at  (0.2,2.8) {$1$};
	\node () at  (-1.6,2.4) {$3$};
	\node () at  (-1.1,2.4) {$1$};
    \node () at  (1.1,2.4) {$1$};
	\node () at  (1.6,2.4) {$3$};
	\node () at  (0.5,0.8) {$2$};
	\node () at  (-0.5,0.8) {$2$};
	\node () at  (0,0.3) {$3$};
	\node () at  (0,-0.6) {$1$};
	
	\draw[-latex] (3) to[out=80,in=-40] (4);
	\draw[-latex] (4) to[out=-60,in=100] (3);
	\draw[-latex] (4) to[out=-140,in=100] (2);
	\draw[-latex] (2) to[out=80,in=-120] (4);
	\draw[latex-latex] (1) to (4);
	
	\draw[-latex] (3) to[out=180,in=0] (2);
	\draw[-latex] (2) to[out=-20,in=-160] (3);

	\draw[latex-latex] (1) --(3);
	
	\draw[latex-latex] (1) -- (2);

	\end{tikzpicture}
	\end{center}
	Representation of $Q(4)$ as a weighted directed graph.
\end{minipage}
\end{tabular}

\begin{proposition}\label{p:qmnotdbe}
 The quasi-metric space $Q(4)$ has three lines none of which is universal. Moreover, its betweenness is neither isomorphic to that of a metric space, nor to those of a quasi-metric space with distances in $\{0,1,2,3\}$, nor to those of a quasi-metric space defined by a directed graph. 
\end{proposition}
\begin{proof}
One can check that the betweenness of $Q(4)$ is given by:
$$\B(Q(4))=\{pqr,rpq,sqp,qps\}.$$

Then, the lines of $Q(4)$ are given by: 
$$\ovd{pq}=\ovd{pr}=\ovd{qr} = \ovd{rp}=\ovd{rq}= \{ p,q,r\},$$ 
$$\ovd{qp}=\ovd{sq}=\ovd{sp}=\ovd{qs}=\ovd{ps} = \{p,q,s\}$$ and 
$$\ovd{rs}=\ovd{sr}=\{r,s\}$$
proving the first statement.

Since $pqr\in \B(Q(4))$ and $rqp\notin \B(Q(4))$, the betweenness $\B(Q(4))$ is not isomorphic to one of a metric space. 

If $\B(Q(4))$ is isomorphic to the betweenness of a  quasi-metric space $Q=(\{a,b,c,d\},\rho)$, with $\rho:\{a,b,c,d\}^2\to \{0,1,2\}$, then without loss of generality we can assume that this isomorphism is given by $f(a)=p$, $f(b)=q$, $f(c)=r$ and $f(d)=s$ and then 
$\B(P)=\{abc,cab,bad,dba\}$. Hence, 
$$\rho(a,b)=\rho(b,a)=\rho(b,c)=\rho(c,a)=\rho(d,b)=\rho(a,d)=1 $$
 and 
 $$\rho(a,c)=\rho(c,b)=\rho(d,a)=\rho(b,d)=2.$$
Since $cad\notin \B(P)$ we have that $\rho(c,d)<\rho(c,a)+\rho(a,d)=2$ and then $\rho(c,d)=1$. This implies the contradiction $bcd\in \B(P)$. Therefore, $\B(Q(4))$ is not isomorphic to the betweenness of a quasi-metric space with distances in $\{0,1,2\}$.

From this we deduce that $\B(Q(4))$ is not isomorphic to the betweenness of a quasi-metric space defined by a directed graph. Indeed, the diameter of a directed graph with four vertices is at most three. When it is less than three, the directed graph defines a quasi-metric space with distances in $\{0,1,2\}$ and we have shown that in this case its betweenness is not isomorphic to $\B(Q(4))$. When it is three, then it defines a quasi-metric space with a universal line and thus its betweenness can not be isomorphic to $\B(Q(4))$ either.
\end{proof}

We now prove our main result.

\begin{theorem}\label{p:unique}
 Let $P=(\{a,b,c,d\},\rho)$ a quasi-metric space without universal lines and with less than four lines, then $\B(P)$ is isomorphic to $\B(Q(4))$.
\end{theorem}
\begin{proof}
Let $\B=\B(P)$ be the betweenness of $P$.

We can assume that $\B \neq \emptyset$ as otherwise, for each pair of points $x,y$ of $P$ the line $\ovd{xy}$ contains only $x$ and $y$ and then $P$ would have six different lines. 

Without loss of generality we assume that $abc\in \B$. Then, $a\in \ovd{bc}$, $b\in \ovd{ac}$ and $c\in \ovd{ab}$. 
Since $P$ has four points and no universal line, we deduce that  
$$\ovd{ab}=\ovd{ac}=\ovd{bc}=\{a,b,c\}=:\ell_1 .$$ 

The point $d$ does not belong to any of these lines. 
In terms of the betweenness this fact can be expressed as follows.
\begin{eqnarray}
d \notin \ovd{ab} \iff \{dab,adb,abd\} \cap \B = \emptyset \label{eq:ab} \\
d \notin \ovd{bc} \iff \{dbc,bdc,bcd\} \cap \B = \emptyset \label{eq:bc} \\
d \notin \ovd{ac} \iff \{dac,adc,acd\} \cap \B = \emptyset. \label{eq:ac}
\end{eqnarray}
If the point $d$ appears in no triple of $\B$, then each line $\ovd{dx}$ contains only $d$ and $x$, for each $x\in\{a,b,c\}$, which would imply that $P$ has more than three lines. Thus, there is a triple in $\B$ containing the point $d$. From (\ref{eq:ab}), (\ref{eq:bc}) and (\ref{eq:ac}) the only options for this triple are those associated to $d\in \ovd{ca}\cup \ovd{cb}\cup \ovd{ba}$. 

We first prove that $d\notin \ovd{ca}$. For the sake of a contradiction, let us assume that $d\in \ovd{ca}$. Then, $\ovd{ca}\neq \ell_1$ and since $P$ has no universal line we get that $$\ell_2:=\ovd{ca}=\{a,c,d\}.$$ 

Moreover, $b\notin \ell_2=\ovd{ca}$ implies 
$\ovd{db}=\ovd{bd}\notin \{\ell_1,\ell_2\},$ since $d\notin \ell_1$. It also implies that $a\notin \ovd{cb}$, since $cab,cba$ are not in $\B$, and, as $abc\in \B$, we also know that $acb\notin \B$. 

From the fact that $a\notin \ovd{cb}$, we get that $\ovd{cb}\notin \{\ell_1, \ell_2\}$ and then,
$$\ell_3=\{b,c,d\}=\ovd{cb}=\ovd{bd}=\ovd{db}.$$
This implies that $c\in \ovd{bd}$ which in turns implies that $cbd\in \B$, since $d\notin \ovd{bc}$. 

As $a\notin \ell_3$ we get that $\ell_2=\ovd{ad}=\ovd{da}=\ovd{ca}$. This implies that $c\in \ovd{ad}$ which, as before, implies that $cad\in \B$, since $d\notin \ovd{ac}$. But, $cbd,cad\in \B$ implies that $\ovd{cd}=\{a,b,c,d\}$ which establishes the contradiction. 

We know that $d\notin \ovd{ac}$ and we have just proved that $d\notin \ovd{ca}$. Then, no triple in $\B$ contains $a$, $c$ and $d$. Thus, $c\notin \ovd{ad}\cup \ovd{da}$ and $a\notin \ovd{cd}\cup \ovd{dc}$ which implies that the three lines in $P$ are 
\begin{eqnarray*}
\ell_1& =& \{a,b,c\}=\ovd{ab}=\ovd{ac}=\ovd{bc}=\ovd{ca} \\
\ell_2 & = & \ovd{da}=\ovd{ad}, c\notin \ell_2 \\
\ell_3 & = & \ovd{dc}=\ovd{cd}, a\notin \ell_3
 \end{eqnarray*}
 
Our analysis now splits into two cases: $b\in \ell_2$ or $b\in \ell_3$. 

If $b\in \ell_2=\ovd{ad}$, then at least one of the triples $bad,abd$ or $adb$ belongs to $\B$. We know that $d\notin \ovd{ab}$ which implies that the triples $dab,adb,abd$ are not in $\B$. Thus, we get that $bad\in \B$. Similarly, as $\ell_2$ is also equal to $\ovd{da}$ we get that the triple $dba$ belongs to $\B$ and that the triples $bda$ and $dab$ does not belong to $\B$. Hence $bad$ and $dba$ are the only triples in $\{a,b,d\}^3$ which belong to $\B$. 
From the fact that the triple $bad$ belongs to $B$ we get that $\ovd{ba}=\ovd{bd}=\ovd{db}=\ell_2$.

Now, $b\in \ovd{ca}$ if and only if some of the triples $bca,cba$ or $cab$ belongs to $\B$. But, as $c\notin \ell_2=\ovd{ba}$ we get that the triple $cab$ is in $\B$ and that the triples $cba,bca$ are not in $\B$. Since we are assuming that $abc\in \B$ we also get that the triples $bac$ and $acb$ are not in $\B$. Whence, we get that $\B\cap \{a,b,c\}^3=\{abc,cab\}$. 

As above, from the fact that the triple $cab$ belongs to $\B$ we get that $a\in \ovd{cb}=\{a,b,c\}=\ell_1$. Then, $d\notin \ovd{cb}\cup \ovd{bc}$ which proves that $\B$ contains no triple containing $b,c$ and $d$. This shows that $\ell_3=\{c,d\}$ and that 
$$\B=\{abc,bad,cab,dba\}.$$
It is easy to see that the mapping associating
$p$ to $a$, $q$ to $b$, $r$ to $c$ and $s$ to $d$ shows that $\B(Q(4))$ and $\B$ are isomorphic.

A similar argument can be applied when $b\in \ell_3=\ovd{dc}$. In this situation we get that $\B\cap \{b,c,d\}^3=\{dcb,cbd\}$ which implies that $\ovd{cb}=\{b,c,d\}=\ell_3$. We also get that $\B\cap \{a,b,c\}^3=\{abc,bca\}$ which shows that $\ovd{ba}=\{a,b,c\}=\ell_1$ from which we conclude that $\B$ contains no triple containing $a$ and $d$. 
Therefore, $\B=\{abc,bca,cbd,dcb\}$ and the mapping associating $p$ to $b$, $q$ to $c$, $r$ to $a$ and $s$ to $d$ shows that $\B$ and $\B(Q(4))$ are indeed isomorphic.

\end{proof}

\begin{corollary} The DBE property is valid for every quasi-metric space on four points which is metric, has distances in $\{0,1,2\}$ or is defined by a directed graph.
\end{corollary}
\begin{proof}
 Direct from previous theorem and Proposition \ref{p:qmnotdbe}.
\end{proof}


\begin{thebibliography}{10}

\bibitem{ABMZ}
P. Aboulker, L. Beaudou, M. Matamala, J. Zamora:
Graphs with no induced house nor induced hole have the de Bruijn-Erdős property. {J. Graph Theory}, 100(4), 638--652, 2022.

\bibitem{metricSpace}
{P.~Aboulker, X.~Chen, G.~Huzhang, R.~Kapadia and C.~Supko.}
\emph{Lines, Betweenness and Metric Spaces. Discrete \& Computational Geometry} 
\textbf{56}(2)  (2016), 427-448.


\bibitem{AK}
{P. Aboulker and R. Kapadia}, 
{The Chen-Chv{\'{a}}tal conjecture for metric spaces induced by distance-hereditary graphs},
\emph{Eur. J. Comb.}, \textbf{43} (2015), {1--7}.

\bibitem{AMRZ}
{P. Aboulker, M. Matamala, P. Rochet, J. Zamora}, 
{A new class of graphs that satisfies the Chen-Chvátal Conjecture.} 
\emph{J. of Graph Theory.} \textbf{87}(1) (2018), 77--88.

\bibitem{AM20}
{G. Araujo-Pardo, M. Matamala},
{Chen and Chvátal's Conjecture in tournaments}
\emph{Eur. J. Comb.} 97:103374 (2021).

\bibitem{BBCCCCFZ} 
{L. Beaudou, A. Bondy, X. Chen, E. Chiniforooshan, M.  Chudnovsky,
V.  Chv{\'{a}}tal, N. Fraiman and Y.  Zwols},
{A De Bruijn-Erd{\H{o}}s Theorem for Chordal Graphs},
\emph{Electr. J. Comb.}, \textbf{22}(1), {P1.70},  {2015}.

\bibitem{BKR}
L. Beaudou, G. Kahn and M. Rosenfeld.
Bisplit graphs satisfy the Chen-Chvátal conjecture. Discret. Math. Theor. Comput. Sci. 21(1) (2019), \# 5.

\bibitem{CC}
X. Chen and V. Chv\'{a}tal,
Problems related to a de Bruijn - Erd\H os theorem,
\emph{Discrete Applied Mathematics} \textbf{156} (2008),  2101--2108.

\bibitem{ChCh11}
E. Chiniforooshan and  V. Chv{\'{a}}tal, 
 A de Bruijn - Erd{\H{o}}s theorem and metric spaces,
 \emph{Discrete Mathematics {\&} Theoretical Computer Science},
 \textbf{13} (1) (2011), {67--74}.

\bibitem{Chvatal2}
V. Chv\'atal,
A de Bruijn-Erd\H{o}s theorem for 1-2 metric spaces,
 \emph{Czechoslovak Mathematical Journal} \textbf{64} (1) (2014), 45--51.

\bibitem{chvatal_2018} V. Chvátal. A de Bruijn-Erd\H{o}s Theorem
  in Graphs? In R. Gera, T. Haynes and S. Hedetniemi (eds) {\em
    Graph Theory. Favorite Conjectures and Open Problems.}
  Problem Books in Mathematics, Springer, 2018.
  A de Bruijn-Erd\H{o}s theorem for 1-2 metric spaces,

  
 \bibitem{dbe}
N.~G.~De~Bruijn, P.~Erd\H{o}s,
On a combinatorial problem,
{\em Indagationes Mathematicae\/} {\bf 10} (1948),
421--423.

\bibitem{erdos1943}
P.~Erd\H{o}s, Three points collinearity,
\emph{American Mathematical Monthly} \textbf{50} (1943), p.65 with solutions in Volume \textbf{51} (1944) 169--171. 

\bibitem{KP2013}
I. Kantor, ans B. Patkós, Towards a de Bruijn-Erdős theorem in the $L_1$-metric. \emph{Discrete Comput.
Geom.} 49: 659--670 (2013).
 
\bibitem{MZ2020} M. Matamala and J. Zamora.  Lines in bipartite graphs and in 2-metric spaces. {\em Journal of
    Graph Theory} 95(4): 565-585 (2020).

    \bibitem{menger1928} K. Menger. Untersuchungen über allgemeine Metrik. Math. Ann. 100, 75–163 (1928).

\bibitem{SS}
R. Schrader and L. Stenmans.
A de Bruijn-Erdös Theorem for $(q,q-4)$-graphs.
Disc. Applied. Maths. 279 (2020) 198:201.

 \bibitem{Wilson1931}
W. Wilson, On quasi-metric spaces, {\em American Journal of Mathematics\/} {\bf 53}(3) (1931), 675-684. 
\end{thebibliography}
\end{document}